\documentclass[12pt, reqno]{amsart}
\usepackage{amsmath, amsthm, amscd, amsfonts, amssymb, graphicx, color}
\usepackage{yhmath}
\usepackage[bookmarksnumbered, colorlinks, plainpages]{hyperref}
\hypersetup{colorlinks=true,linkcolor=red, anchorcolor=green, citecolor=cyan, urlcolor=red, filecolor=magenta, pdftoolbar=true}

\newtheorem{theorem}{Theorem}[section]
\newtheorem{lemma}[theorem]{Lemma}

\newtheorem{corollary}[theorem]{Corollary}
\theoremstyle{definition}

\theoremstyle{remark}
\newtheorem{remark}[theorem]{Remark}
\numberwithin{equation}{section}

\begin{document}
\title[Sharp Inequalities for the Numerical Radius]{Sharp Inequalities for the Numerical Radius of Block Operator Matrices}
\author[M. Ghaderi Aghideh, M. S. Moslehian, J. Rooin]{M. Ghaderi Aghideh$^{1}$, M. S. Moslehian$^{2}$, and J. Rooin$^{3}$}

\address{$^{1}$ Department of Mathematics, Institute for Advanced Studies in Basic Sciences (IASBS), Zanjan 45137-66731, Iran;\newline Tusi Mathematical Research Group (TMRG), Mashhad, Iran}
\email{m.ghaderiaghideh@iasbs.ac.ir}

\address{$^{2}$ Department of Pure Mathematics, Center of Excellence in Analysis on Algebraic Structures (CEAAS), Ferdowsi University of
Mashhad, P. O. Box 1159, Mashhad 91775, Iran}
\email{moslehian@um.ac.ir, moslehian@member.ams.org}

\address{$^{3}$ Department of Mathematics, Institute for Advanced Studies in Basic Sciences (IASBS), Zanjan 45137-66731, Iran}
\email{rooin@iasbs.ac.ir}



\begin{abstract}
In this paper, we present several sharp upper bounds for the numerical radii of the diagonal and off-diagonal parts of
the $ 2 \times 2$ block operator matrix 
$ \begin{bmatrix}
A& B \\
C& D \end{bmatrix} $. Among extensions of some results of Kittaneh et al., it is shown that if 
$ T= \begin{bmatrix}
A& 0 \\
0& D \end{bmatrix}$, and $f$ and $g$ are non-negative continuous functions on $ [0, \infty ) $ such that $ f(t)g(t)=t \ (t \geq 0) $,
then for all non-negative nondecreasing convex functions $h$ on 
$ [0, \infty ) $, we obtain that
\begin{align*}
& h\left( w^r(T)\right) \\ 
& \leq \max \left( 
\left\| \dfrac{1}{p} h\left(f^{pr}(\left| A \right| )\right)+ 
\dfrac{1}{q}h\left(g^{qr}(\left| A^* \right| )\right)\right\|
,\left\|\dfrac{1}{p} h\left( f^{pr}(\left| D \right| )\right)+ 
\dfrac{1}{q}h\left(g^{qr}(\left| D^* \right| )\right)\right\|\right),
\end{align*} 
where  $ p, q > 1 $ with 
$\dfrac{1}{p} + \dfrac{1}{q} =1$, and $ r \min (p, q )\geq 2 $. 
\end{abstract}

\maketitle

\section{Introduction}
Let 
$\left(\mathcal{H}, \left\langle \cdot, \cdot\right\rangle \right)$ be a complex Hilbert space, and let $\mathbb{B}(\mathcal{H})$
denote the $C^*-$algebra of all bounded linear operators on $\mathcal{H}$. The spectral radius and the numerical radius of an operator $ A \in\mathbb{B}(\mathcal{H}) $ are defined by 
$\rho(A)= \sup\left\lbrace \left| \lambda\right| : \lambda \in  {\rm sp}(A) \right\rbrace$ 	and 
\begin{equation*}
w(A)= \sup \left\lbrace \left| \left\langle Ax, x \right\rangle \right|: x \in \mathcal{H}, \ \left\| x\right\|=1 \right\rbrace, 
\end{equation*}
respectively. It is well known that $\rho(A) \leq w(A)$ and $w(\cdot)$ defines a norm on $\mathbb{B}(\mathcal{H})$, which is equivalent to the usual operator norm $ \left\| \cdot \right\| $; more precisely, 
\begin{equation}\label{inequality 1.2}
\dfrac{1}{2} \left\|A \right\| \leq w(A) \leq \left\|A \right\|
\end{equation}
for any 
$ A \in\mathbb{B}(\mathcal{H}) $. The inequalities in (\ref{inequality 1.2}) are sharp; the second inequality becomes an equality, e.g., if $ A $ is normal, while the first one becomes an equality, e.g., if $ A^2 = 0$. \\ 
An important inequality for $ w(A) $ is the power inequality stating that
\begin{equation*}\label{inequality 1.3}
w(A^n) \leq w(A)^n \qquad (n=1, 2, \ldots).
\end{equation*}
The quantity $w(A)$ is useful in the study of perturbation, convergence, and approximation problems. For more information see \cite{1, 4, 3, 13, 16}. \\
Let $ A, B, C,$ and $D $ be in $\mathbb{B}(\mathcal{H})$. 
We call 
$ \begin{bmatrix}
A& 0 \\
0& D \end{bmatrix} $ 
and 
$ \begin{bmatrix}
0& B \\
C& 0 \end{bmatrix} $
the diagonal and off-diagonal parts of the block matrix
$ \begin{bmatrix}
A& B \\
C& D \end{bmatrix}$, respectively.\\
Hirzallah, Kittaneh, and Shebrawi \cite{4} proved that
\begin{equation}\label{inequality 1.4}
w\left(\begin{bmatrix}
0& B \\
C& 0 \end{bmatrix} \right) 
\leq\dfrac{1}{2}\left(\left\|B\right\| + \left\|C \right\|\right),
\end{equation}
for $ B, C \in\mathbb{B}(\mathcal{H})$. Kittaneh \cite{7,8} showed the following precise estimates of 
$ w(A)$:
\begin{equation} \label{inequality 1.5}
w(A) \leq \dfrac{1}{2} \left\| \, \left|A \right| + \left|A^* \right| \, \right\| 
\end{equation}
and
\begin{equation}\label{inequality 1.6}
\dfrac{1}{4} \left\|\,\left| A \right|^2 
+ \left| A^* \right|^2 \right\| \leq w^2(A) \leq \dfrac{1}{2} \left\|\,\left| A \right|^2 
+ \left| A^* \right|^2 \right\|,
\end{equation}
where $ \left| A \right| = (A^* A)^{\frac{1}{2}} $ denotes the absolute value of $ A $. \\
Also, El-Haddad and Kittaneh \cite{12} established that if $ A \in\mathbb{B}(\mathcal{H})$ and $ A= B + i C $ is the Cartesian decomposition of $A$, then
\begin{equation} \label{inequality 1.7}
2^{-\frac{r}{2}- 1}\left\| \left| B+ C \right|^r + \left| B- C \right|^r\right\|
\leq w^r(A) \leq \dfrac{1}{2} 
\left\| \left| B+ C \right|^r + \left| B- C \right|^r\right\|, 
\end{equation}
for all $ r \geq 2$.\\
The purpose of this paper is to present some general inequalities involving powers of the numerical radius for the diagonal and off-diagonal parts of $ 2 \times 2 $ block operator matrices. 
As a consequence, we generalize inequalities 
(\ref{inequality 1.4}), (\ref{inequality 1.5}), and second inequalities in (\ref{inequality 1.6}) and (\ref{inequality 1.7}). 
\section{Inequalities for the off-diagonal part}
To achieve our results, we need the functional calculus (see, e.g. \cite{17}) and the following lemmas.
The first lemma is a consequence of the classical Young and H\"{o}lder inequalities.
\begin{lemma}\cite[p. 100 and 127]{6}
For $ a, b \geq 0$ and $p,q > 1$ such that 
$\dfrac{1}{p} + \dfrac{1}{q} =1$,\\
$ (a) \ \ ab \leq \dfrac{a^{p}}{p}+ \dfrac{b^q}{q}\leq \left( \dfrac{a^{pr}}{p}+ \dfrac{b^{qr}}{q}\right)^{\frac{1}{r}}$ 
for $r \geq 1$,\\
$ (b) \ \ a_1 b_1 +a_2 b_2 + \cdots + a_n b_n \leq \left( a_1^p + a_2^p+ \cdots + a_n^p \right)^{\frac{1}{p}} \left( b_1^q + b_2^q + \cdots + b_n^q \right)^{\frac{1}{q}} $.
\end{lemma}
The second lemma is an operator version of the classical Jensen inequality.
\begin{lemma}\label{lemma 2.2}\cite[Theorem 1.2]{14}
Let $ A \in\mathbb{B}(\mathcal{H})$ be a self-adjoint operator with $ {\rm sp}(A) \subseteq [m, M] $ for some scalars $ m \leq M $, and let $ x \in \mathcal{H}$ be a unit vector. If $ f(t)$ is a convex function on $[m, M]$, then
\begin{equation*}
f\left( \left\langle Ax, x \right\rangle \right) \leq \left\langle f(A)x, x \right\rangle.
\end{equation*}
In particular, if $A\geq 0$, then
\begin{equation*}
\left\langle Ax, x \right\rangle^r \leq \left\langle A^r x, x \right\rangle \qquad (r \geq 1).
\end{equation*}
\end{lemma}
The third lemma is known as the generalized mixed Cauchy--Schwarz inequality.
\begin{lemma}\cite{5}
Let $ A \in\mathbb{B}(\mathcal{H})$, and let $ x, y \in \mathcal{H}$ be any vectors. If $ f$ and $g $ are non-negative continuous functions on $\left[0, \infty \right) $ satisfying 
$ f(t) g(t) = t \ (t \geq 0), $ then 
\begin{equation*}
\left| \left\langle Ax, y \right\rangle \right|\leq 
\left\langle f^2(\left| A \right|) x, x\right\rangle^{\frac{1}{2}} \left\langle g^2(\left| A^*\right|)y, y \right\rangle^{\frac{1}{2}}.
\end{equation*}
\end{lemma}
The fourth lemma can be found in \cite{9, 8}. 
\begin{lemma}\label{lemma 2.4}
Let $A, B,$ and $D$ be operators in $\mathbb{B}(\mathcal{H})$. Then \\
$(a) \ \ w(A) = \max_{\theta \in \mathbb{R}} \left\| Re \left(e^{i \theta}A \right)\right\|, 
$\\
$(b) \ \ w \left(\begin{bmatrix}
A& 0 \\
0 & D \\
\end{bmatrix}\right)= \max\left( w(A), w(D) \right), $ \\
$(c) \ \ w \left(\begin{bmatrix}
A& B \\
B & A \\
\end{bmatrix}\right) = \max\left( w(A+B), w(A-B) \right), $ \\
$(d) \ \ w \left(\begin{bmatrix}
A& B \\
-B & A \\
\end{bmatrix}\right) = \max\left( w(A+iB), w(A-iB)\right). $
\end{lemma}

The following result is a variant of a known result (see \cite[Corollary 3.5]{KOS}) but with a different proof.
\begin{lemma}\label{rooin1}
Let $h$ be a non-negative nondecreasing convex function on $[0, \infty )$ and let $ A, B  \in  \mathbb{B}(\mathcal{H}) $ be positive operators. Then
\[h\left( \left\| \dfrac{A+B}{2}\right\|  \right) \leq \left\| \dfrac{h(A) + h(B)}{2}\right\|.\]
\end{lemma}
\begin{proof}
For each unit vector $ x \in \mathcal{H} $, we have
\begin{align}
h\left( \left\langle  \dfrac{A+B}{2} x, x\right\rangle  \right) 
& = h \left(  \dfrac{\left\langle A x, x \right\rangle + \left\langle Bx ,x \right\rangle  }{2}\right) \nonumber\\
&  \leq \dfrac{h\left(\left\langle Ax ,x \right\rangle  \right) +h\left(\left\langle Bx ,x \right\rangle  \right)}{2} \tag{by the convexity of $ h $} \nonumber\\
& \leq \dfrac{\left\langle h(A) x, x \right\rangle + \left\langle h(B) x, x \right\rangle }{2} \tag{by the operator Jensen inequality} \nonumber\\
& = \left\langle  \dfrac{h(A)+ h(B)}{2} x , x \right\rangle \nonumber \\
& \leq \left\| \dfrac{h(A) +h(B)}{2}\right\|\label{msm20}. 
\end{align}
Now, since $h$ is a non-negative, non-decreasing and convex (continuous) function, by considering (\ref{msm20}) and taking the  supermum from the left hand side, we get  
\begin{align*}
h\left( \left\| \dfrac{A+B}{2}\right\|  \right) & = h\left( w\left( \dfrac{A+B}{2}\right)  \right) \\
& = h\left(\sup \left\langle  \dfrac{A+B}{2} x, x\right\rangle  \right) \\
& = \sup \left( h\left( \left\langle  \dfrac{A+B}{2} x, x\right\rangle \right) \right) \\
& \leq \left\| \dfrac{h(A) + h(B)}{2}\right\|.
\end{align*}
\end{proof}

We are in a position to demonstrate the main results of this section by adopting and extending some techniques of {10,msmz, 12, 11}.
The following theorem gives a generalization of inequality (\ref{inequality 1.4}). Recall that the polarization identity says that for any 
elements $ x, y$ of an inner product space $\mathcal{H} $, 
\begin{equation*}
\left\langle x, y \right\rangle = 
\dfrac{1}{4} \sum_{k=0}^{3} i^k \left\| x+ i^k y \right\|^2. 
\end{equation*}

\begin{theorem}\label{theorem 2.5}
Let 
$ S= \begin{bmatrix}
0& B \\
C& 0 \end{bmatrix} \in\mathbb{B}(\mathcal{H} \oplus \mathcal{H})$, and let $f$ and $ g$ be non-negative continuous functions on 
$ [0, \infty )$ such that $ f(t)g(t)=t \ (t \geq 0) $.
Then for all non-negative nondecreasing convex functions $h$ on 
$ [0, \infty ) $,
\begin{equation}\label{inequality 2.1}
h\left( w(S)\right) \leq 
\dfrac{1}{4}\left\| h\left( f^2(\left| B \right| )\right)+ 
h\left(g^2(\left| B \right|)\right)\right\|+
\dfrac{1}{4}\left\| h\left(f^2(\left| C \right|)\right)+ 
h\left(g^2(\left| C \right|)\right)\right\|.
\end{equation}
\end{theorem}
\begin{proof}
Let $ B= U \left| B \right| $, and let $ C= V \left| C \right| $ be the polar decompositions of the operators $B$ and $C$. Then 
\begin{equation*}
S= W \left| S \right|=
\begin{bmatrix}
0& U \\
V& 0 \end{bmatrix}
\begin{bmatrix}
\left| C \right|& 0 \\
0 & \left| B \right| \end{bmatrix}
\end{equation*}
is the polar decomposition of $ S $.
Let $x=(x_1, x_2)$ be any unit vector in 
$\mathcal{H} \oplus \mathcal{H}$; that is, $ \left\| x_1\right\| ^2 + \left\| x_2\right\|^2=1$. 
Then for all $ \theta \in \mathbb{R}$, we obtain 
\begin{align*}
& Re\left\langle e^{i \theta} Sx, x \right\rangle \\
& = Re \left\langle e^{i \theta} W \left| S \right| x, x \right\rangle\\
& = Re \left\langle e^{i \theta} W f(\left| S \right|) g(\left| S \right|) x, x \right\rangle 
\tag{by functional calculus} \\
& = Re \left\langle e^{i \theta} g(\left| S \right|) x, f(\left| S \right|) W^* x \right\rangle\\
& = Re\left\langle e^{i \theta}
\begin{bmatrix}
g(\left| C \right|) & 0 \\
0 & g(\left| B \right|) 
\end{bmatrix}
\begin{bmatrix}
x_1\\x_2
\end{bmatrix}, 
\begin{bmatrix}
f(\left| C \right|) & 0 \\
0 & f(\left| B \right|) 
\end{bmatrix}
\begin{bmatrix}
0 & V^*\\
U^* & 0
\end{bmatrix}
\begin{bmatrix}
x_1\\x_2
\end{bmatrix} \right\rangle \\
& = Re\left\langle e^{i \theta}\left(g(\left| C \right|)x_1, g(\left| B \right|) x_2 \right), \left( f(\left| C \right|)V^* x_2, f(\left| B \right|) U^* x_1 \right) \right\rangle \\
& =Re \left(\left\langle e^{i \theta}g(\left| C \right|) x_1, f(\left| C \right|) V^* x_2\right\rangle + \left\langle e^{i \theta} g(\left| B \right|) x_2, f(\left| B \right|) U^* x_1 \right\rangle \right)\\
& = \dfrac{1}{4} \left( \left\| e^{i \theta}g(\left| C \right|) x_1 + f(\left| C \right|) V^* x_2 \right\|^2 - \left\| e^{i \theta}g(\left| C \right|) x_1- f(\left| C \right|) V^* x_2 \right\|^2 \right) \\
& \quad +\dfrac{1}{4} \left( \left\| e^{i \theta} g(\left| B \right|) x_2+ 
f(\left| B \right|) U^* x_1 \right\|^2 - \left\| e^{i \theta} g(\left| B \right|) x_2- f(\left| B \right|) U^* x_1 \right\|^2 \right) 
\tag{by the polarization identity} \\
& \leq \dfrac{1}{4}\left\| e^{i \theta}g(\left| C \right|) x_1 + f(\left| C \right|) V^* x_2 \right\|^2
+ \dfrac{1}{4}\left\| e^{i \theta} g(\left| B \right|) x_2+ 
f(\left| B \right|) U^* x_1 \right\|^2 \\
& = \dfrac{1}{4} \left\|
\begin{bmatrix} 
e^{i \theta}g(\left| C \right|) \quad
f(\left| C \right|) V^*
\end{bmatrix}
\begin{bmatrix}
x_1\\ x_2
\end{bmatrix}
\right\|^2 
+ \dfrac{1}{4} \left\|
\begin{bmatrix} 
f(\left| B \right|) U^* \quad
e^{i \theta}g(\left| B \right|) 
\end{bmatrix}
\begin{bmatrix}
x_1\\ x_2
\end{bmatrix}
\right\|^2 \\
& \leq \dfrac{1}{4}\left\|
\begin{bmatrix} 
e^{i \theta}g(\left| C \right|) \quad
f(\left| C \right|) V^*
\end{bmatrix}\right\| ^2
+ \dfrac{1}{4}\left\|
\begin{bmatrix} 
f(\left| B \right|) U^* \quad
e^{i \theta}g(\left| B \right|) 
\end{bmatrix}\right\| ^2 \\
& = \dfrac{1}{4} \left\|
\begin{bmatrix} 
e^{i \theta}g(\left| C \right|) \quad
f(\left| C \right|) V^*
\end{bmatrix} 
\begin{bmatrix} 
e^{-i \theta}g(\left| C \right|) \\
V f(\left| C \right|) 
\end{bmatrix}
\right\|
+ \dfrac{1}{4} \left\|
\begin{bmatrix} 
f(\left| B \right|) U^* \quad
e^{i \theta}g(\left| B \right|) 
\end{bmatrix}
\begin{bmatrix} 
U f(\left| B \right|) \\
e^{-i \theta}g(\left| B \right|) 
\end{bmatrix} \right\| \\
& = \dfrac{1}{4} \left\|g^2(\left| C \right|)+ f(\left| C \right|)V^* V f(\left| C \right|)\right\|
+ \dfrac{1}{4} \left\|f(\left| B \right|) U^*U f(\left| B \right|)+ g^2(\left| B \right|) \right\| \\
& = \dfrac{1}{4}\left\|f^2(\left| C \right|)+g^2( \left| C \right|)\right\|
+\dfrac{1}{4} \left\|f^2(\left| B \right|)+ g^2(\left| B \right|) \right\|.
\end{align*}
Taking the supremum over all unit vectors $ x=(x_1, x_2)$ and utilizing Lemma \ref{lemma 2.4} (a), we get
\begin{equation*}
w(S) \leq \dfrac{1}{4} \left\|f^2\left( \left|C \right| \right) +g^2\left(\left| C \right|\right) \right\|
+ \dfrac{1}{4} \left\|f^2\left(\left| B \right| \right) + g^2\left(\left| B \right|\right)\right\|.
\end{equation*}
Therefore, 
\begin{align*}
h(w(S)) &
\leq \dfrac{1}{2} h\left(\left\|\dfrac{f^2(\left| C \right|)+ g^2( \left| C \right|)}{2}\right\|\right) 
+\dfrac{1}{2} h \left( \left\| \dfrac{f^2(\left| B \right|)
+g^2(\left| B \right|)}{2}
\right\| \right) 
\tag{since \textit{h} is nondecreasing and convex } \\
& \leq \dfrac{1}{4} \left\| h \left( f^2(\left| C \right|)\right) 
+ h \left( g^2( \left| C \right|)\right) \right\|
+\dfrac{1}{4} \left\| h\left( f^2(\left| B \right|)\right) +h\left( g^2(\left| B \right|)\right) \right\|
\tag{by Lemma \ref{rooin1}.}
\end{align*}
\end{proof}
The next corollary gives a generalization of inequality (\ref{inequality 1.4}). 
\begin{corollary} 
Let $ B, C \in\mathbb{B}(\mathcal{H}) $. Then
\begin{equation}\label{inequality 2.2}
w^r \left( \begin{bmatrix}
0& B \\
C& 0 \end{bmatrix}\right) 
\leq \dfrac{1}{4}\left\|\,\left|B\right|^{2r\alpha}
+\left|B\right|^{2r(1-\alpha)}\right\|
+ \dfrac{1}{4}\left\|\,\left|C \right|^{2r\alpha} + \left|C\right|^{2r(1-\alpha)} \right\|,
\end{equation}
for all $ \alpha\in[0, 1]$ and $ r \geq 1 $.
\end{corollary}
\begin{proof}
Inequality (\ref{inequality 2.2}) follows from inequality (\ref{inequality 2.1}) by putting $ h(t)=t^r $, $f(t)= t^\alpha $, and $ g(t)= t^{1-\alpha}$. 
\end{proof}
\begin{remark}
Let $ B, C \in\mathbb{B}(\mathcal{H}) $. The following lower bound was obtained in \cite{10}. 
\begin{equation*}\label{inequality 2.3}
w^{\frac{1}{2}}(BC) \leq w \left( \begin{bmatrix} 
0& B \\
C& 0 \end{bmatrix}\right). 
\end{equation*}
Also, in the same paper, it was shown that if $ B, C \geq 0$, then
\begin{equation*}
\left\|B^{\frac{1}{2}} C^{\frac{1}{2}} \right\|^2 
= \rho(BC)
\left( \leq w(BC) \right).
\end{equation*}
\end{remark}
\begin{corollary}
Let $ B, C \in\mathbb{B}(\mathcal{H})$, and let $ C $ be normal. Then
\begin{equation*}
\left\|B+ C \right\|^r \leq 2^{r-2} 
\left(\left\|\,\left|B\right|^{2r\alpha}
+\left|B\right|^{2r(1-\alpha)}\right\|
+\left\|\,\left|C \right|^{2r\alpha} + \left|C\right|^{2r(1-\alpha)} \right\|\right), 
\end{equation*}
for all $ \alpha \in [0, 1]$ and $ r \geq 1$.
\end{corollary}
\begin{proof}
We have
\begin{align*}
\left\|B+ C\right\|^r & = \left\|\begin{bmatrix}
0& B \\
C^*& 0 \end{bmatrix}+ \begin{bmatrix}
0& B \\
C^*& 0 \end{bmatrix}^* \right\|^r \\
& = 2^r w^r\left(Re \begin{bmatrix}
0& B \\
C^*& 0 \end{bmatrix} \right) \\
& \leq 2^r \max_{\theta \in \mathbb{R}} \left\| Re\left(e^{i\theta} \begin{bmatrix}
0& B \\
C^*& 0 \end{bmatrix}\right)\right\|^r\\
&= 2^r w^r\left(\begin{bmatrix}
0& B \\
C^*& 0 \end{bmatrix}\right) \tag{by Lemma \ref{lemma 2.4} (a)}\\
& \leq 2^{r-2} \left(\left\|\,\left|B\right|^{2r\alpha}
+\left|B\right|^{2r(1-\alpha)}\right\|
+ \left\|\,\left|C^* \right|^{2r\alpha} + \left|C^*\right|^{2r(1-\alpha)} \right\|\right).
\tag{by inequality (\ref{inequality 2.2})} 
\end{align*}
Since $ C $ is normal, $ \left| C\right| = \left| C^* \right| $, and the proof is complete. 
\end{proof}
\begin{theorem} \label{theorem 2.9}
Let 
$ S= \begin{bmatrix}
0& B \\
C& 0 \end{bmatrix} \in\mathbb{B}(\mathcal{H} \oplus \mathcal{H})$, $ r \geq 2 $, and $p, q > 1$ with $\dfrac{1}{p} + \dfrac{1}{q}=1$. If 
$f_1$, $g_1$, $f_2$, and $g_2$ are non-negative continuous functions on $[0, \infty)$ such that $f_1(t)g_1(t)= f_2(t)g_2(t)=t \ (t \geq 0)$, then
\begin{equation}\label{inequality 2.4}
w^r(S) \leq 2^{-\frac{r}{2}-1}{\max}^{\frac{1}{p}} \left( \alpha,
\beta \right) 
{\max}^{\frac{1}{q}} \left(\gamma, \delta \right), 
\end{equation}
and
\begin{equation}\label{inequality 2.5}
w^r(S) \leq 2^{-\frac{r}{2}-1} 
{\max}^{\frac{1}{p}}\left(\alpha^{'}, \beta^{'} \right) 
{\max}^{\frac{1}{q}}\left(\gamma^{'}, \delta^{'}\right), 
\end{equation}
where\\
{\small 
$ \alpha= \left\| f_1^{rp} \left( \left| B^*- iC \right| \right) +
f_2^{rp} \left( \left| B^*+ iC \right| \right)\right\|, \qquad
\beta= \left\| f_1^{rp} \left( \left| B+i C^* \right| \right) + f_2^{rp} \left( \left| B- iC^* \right| \right) \right\| $,\\
$ \gamma= \left\| g_1^{rq} \left( \left| B^*- iC \right| \right) + g_2^{rq} \left( \left| B^*+ iC \right| \right)\right\|, \qquad
\delta= \left\| g_1^{rq} \left( \left| B+i C^* \right| \right) + g_2^{rq} \left( \left| B-i C^* \right| \right)\right\|, $\\
}
and\\
{\small 
$ \alpha^{'} = \left\| f_1^{rp} \left( \left| B^*- iC \right| \right) + g_2^{rp} \left( \left| B^*+ iC \right| \right)\right\|,\qquad
\beta^{'}= \left\| f_1^{rp} \left( \left| B+i C^* \right| \right) + g_2^{rp} \left( \left| B-i C^* \right| \right)\right\| $,\\
$ \gamma^{'}= \left\| g_1^{rq} \left( \left| B^*- iC \right| \right) + f_2^{rq} \left( \left| B^*+ iC \right| \right)\right\|,\qquad
\delta^{'}= \left\| g_1^{rq} \left( \left| B+i C^* \right| \right) + f_2^{rq} \left( \left| B- iC^* \right| \right)\right\|.$
}
\end{theorem}
\begin{proof}
Assume that $ S= S_1 +iS_2 $ is the Cartesian decomposition of $ S $, and that $x$ is any unit vector in 
$\mathcal{H} \oplus \mathcal{H}$. Then
\begin{align*}
\left| \left\langle Sx, x \right\rangle\right|^r 
& = \left|\left\langle \left( S_1 + iS_2\right)x, x \right\rangle \right|^r \\
&=\left(\left\langle S_1 x, x \right\rangle^2 + \left\langle S_2 x,x \right\rangle^2 \right)^{\frac{r}{2}} \\
& = 2^{-\frac{r}{2}}\left(\left\langle(S_1+ S_2)x, x \right\rangle^2 + \left\langle (S_1- S_2)x,x \right\rangle^2 \right)^{\frac{r}{2}} \\
& \leq 2^{-\frac{r}{2}} 2^{\frac{r}{2}-1}\left(\left| \left\langle(S_1 + S_2)x, x \right\rangle \right| ^r +\left|\left\langle (S_1 - S_2)x,x \right\rangle \right| ^r \right)
\tag{by the convexity of $ t^{\frac{r}{2}}$ for $r \geq 2$} \\
& \leq \dfrac{1}{2} \left( \left\langle \left| S_1+ S_2 \right| x, x \right\rangle ^r + \left\langle \left| S_1- S_2 \right| x, x \right\rangle ^r \right). 
\tag{by the convexity of $ \left|t\right|$ and Lemma \ref{lemma 2.2}}
\end{align*}
A straightforward computation shows that
{\footnotesize 
\begin{align*}
& 2 \left| \left\langle Sx, x \right\rangle\right|^r \\
& \leq \left\langle \begin{bmatrix}
\dfrac{1}{\sqrt{2}} \left| B^*-iC \right| & 0 \\
0 & \dfrac{1}{\sqrt{2}}\left| B+iC^* \right| 
\end{bmatrix} x, x \right\rangle^r 
+ \left\langle \begin{bmatrix}
\dfrac{1}{\sqrt{2}}\left| B^*+iC \right| & 0 \\
0 & \dfrac{1}{\sqrt{2}}\left| B-iC^* \right| 
\end{bmatrix} x, x \right\rangle^r.
\end{align*}
}
Hence, 
{\footnotesize 
\begin{align*}
& 2^{\frac{r}{2}+ 1}\left|\left\langle Sx, x\right\rangle\right|^r \\
& \leq \left\langle \begin{bmatrix}
\left| B^*-iC \right| & 0 \\
0 & \left| B+iC^* \right| 
\end{bmatrix} x, x \right\rangle^r 
+ \left\langle \begin{bmatrix}
\left| B^*+iC \right| & 0 \\
0 & \left| B-iC^* \right| 
\end{bmatrix} x, x \right\rangle^r\\
& \leq \left\langle \begin{bmatrix}
f_1^2 \left( \left| B^*-iC \right|\right) & 0 \\
0 & f_1^2 \left( \left| B+iC^* \right| \right) 
\end{bmatrix} x, x \right\rangle^{\frac{r}{2}} 
\left\langle \begin{bmatrix}
g_1^2 \left( \left| B^*-iC \right|\right) & 0 \\
0 & g_1^2 \left( \left| B+iC^* \right| \right) 
\end{bmatrix} x, x \right\rangle^{\frac{r}{2}} \\
& \quad + \left\langle \begin{bmatrix}
f_2^2 \left( \left| B^*+iC \right|\right) & 0 \\
0 & f_2^2 \left( \left| B-iC^* \right| \right) 
\end{bmatrix} x, x \right\rangle^{\frac{r}{2}} 
\left\langle \begin{bmatrix}
g_2^2 \left( \left| B^*+iC \right|\right) & 0 \\
0 & g_2^2 \left( \left| B-iC^* \right| \right) 
\end{bmatrix} x, x \right\rangle^{\frac{r}{2}} 
\tag{by the mixed Cauchy--Schwarz inequality}
\end{align*}
\begin{align*}
&\leq \left\langle \begin{bmatrix}
f_1^r \left( \left| B^*-iC \right|\right) & 0 \\
0 & f_1^r \left( \left| B+iC^* \right| \right) 
\end{bmatrix} x, x \right\rangle
\left\langle \begin{bmatrix}
g_1^r \left( \left| B^*-iC \right|\right) & 0 \\
0 & g_1^r \left( \left| B+iC^* \right| \right) 
\end{bmatrix} x, x \right\rangle \\
& \quad + \left\langle \begin{bmatrix}
f_2^r \left( \left| B^*+iC \right|\right) & 0 \\
0 & f_2^r \left( \left| B-iC^* \right| \right) 
\end{bmatrix} x, x \right\rangle
\left\langle \begin{bmatrix}
g_2^r \left( \left| B^*+iC \right|\right) & 0 \\
0 & g_2^r \left( \left| B-iC^* \right| \right) 
\end{bmatrix} x, x \right\rangle
\tag{by Lemma \ref{lemma 2.2}}\\
& \leq \left\langle \begin{bmatrix}
f_1^{rp} \left( \left| B^*- iC \right| \right) +
f_2^{rp} \left( \left| B^*+ iC \right| \right) & 0 \\
0& f_1^{rp} \left( \left| B+i C^* \right| \right) 
+ f_2^{rp} \left( \left| B- iC^* \right| \right)
\end{bmatrix} x, x \right\rangle^{\frac{1}{p}}\\
& \quad \times \left\langle \begin{bmatrix}
g_1^{rq} \left( \left| B^*- iC \right| \right) + g_2^{rq} \left( \left| B^*+ iC \right| \right)& 0 \\
0& g_1^{rq} \left( \left| B+i C^* \right| \right) + g_2^{rq} \left( \left| B-i C^* \right| \right)
\end{bmatrix} x, x \right\rangle^{\frac{1}{q}}.
\tag{by the H\"{o}lder inequality and Lemma \ref{lemma 2.2}} 
\end{align*}
}
Taking the supremum over all unit vectors $x$ we get inequality (\ref{inequality 2.4}). Inequality 
(\ref{inequality 2.5}) is achieved by a similar argument.
\end{proof}
\section{Inequalities for the diagonal part}
In this section, we obtain some upper bounds for the numerical radius of diagonal operator matrices.
\begin{theorem}\label{theorem 3.1}
Let 
$ T= \begin{bmatrix}
A& 0 \\
0& D \end{bmatrix} \in\mathbb{B}(\mathcal{H} \oplus \mathcal{H})$, and let $f$ and $g$ be non-negative continuous functions on $ [0, \infty )$ such that $ f(t)g(t)=t \ (t \geq 0) $.
Then for all non-negative nondecreasing convex functions $h$ on 
$ [0, \infty ) $, the following inequality holds:
{\small 
\begin{equation}\label{inequality 3.1}
 h(w(T)) \leq 
\dfrac{1}{2} \max \left( \left\| h\left( f^2(\left| A \right|)\right) + 
h\left( g^2(\left| A \right|)\right)\right\|
, \left\| h\left( f^2(\left| D \right| )\right) + 
h\left( g^2(\left| D \right|)\right) \right\| \right).
\end{equation}
}
\end{theorem}
\begin{proof}
Let $ A= U \left| A \right| $ and $ D= V \left| D \right| $ be the polar decompositions of the operators $A$ and $D$. Then 
\begin{equation*}
T= W \left| T \right|=
\begin{bmatrix}
U& 0 \\
0& V \end{bmatrix}
\begin{bmatrix}
\left| A \right|& 0 \\
0 & \left| D \right| \end{bmatrix}
\end{equation*}
is the polar decomposition of $ T $.
Let $x=(x_1, x_2)$ be any unit vector in 
$\mathcal{H} \oplus \mathcal{H}$; 
that is, $ \left\| x_1\right\| ^2 + \left\| x_2\right\|^2=1$. 
Then for all $ \theta \in \mathbb{R}$, we obtain 
\begin{align*}
& Re\left\langle e^{i \theta} Tx, x \right\rangle \\
& \quad= Re \left\langle e^{i \theta} W \left| T \right| x, x \right\rangle\\
& \quad= Re\left\langle e^{i \theta} W f(\left| T \right|) g(\left| T \right|) x, x \right\rangle 
\tag{by functional calculus} \\
& \quad= Re\left\langle e^{i \theta} g(\left| T \right|) x, f(\left| T \right|) W^* x \right\rangle \\
& \quad= Re\left\langle e^{i \theta}
\begin{bmatrix}
g(\left| A \right|) & 0 \\
0 & g(\left| D \right|) 
\end{bmatrix}
\begin{bmatrix}
x_1\\x_2
\end{bmatrix}, 
\begin{bmatrix}
f(\left| A \right|) & 0 \\
0 & f(\left| D \right|) 
\end{bmatrix}
\begin{bmatrix}
U^* & 0\\
0 & V^*
\end{bmatrix}
\begin{bmatrix}
x_1\\x_2
\end{bmatrix}
\right\rangle \\
& \quad= Re\left\langle e^{i \theta}\left( g(\left| A \right|)x_1, g(\left| D \right|) x_2 \right), \left( f(\left| A \right|)U^* x_1, f(\left| D \right|) V^* x_2 \right) \right\rangle \\
& \quad = Re \left(\left\langle e^{i \theta}g(\left| A \right|) x_1, f(\left| A \right|) U^* x_1\right\rangle + \left\langle e^{i \theta} g(\left| D \right|) x_2, f(\left| D \right|) V^* x_2 \right\rangle \right) \\
& \quad = \dfrac{1}{4} \left( \left\| e^{i \theta}g(\left| A\right|) x_1 + f(\left| A \right|) U^* x_1 \right\|^2 - \left\| e^{i \theta}g(\left| A \right|) x_1- f(\left| A \right|) U^* x_1 \right\|^2 \right) \\
& \qquad + \dfrac{1}{4}\left( \left\| e^{i \theta} g(\left| D \right|) x_2+ 
f(\left| D \right|) V^* x_2 \right\|^2 - \left\| e^{i \theta} g(\left| D \right|) x_2- f(\left| D \right|) V^* x_2 \right\|^2 \right) 
\tag{by the polarization identity} 
\end{align*}
\begin{align*}
& \quad \leq \dfrac{1}{4}\left\| e^{i \theta}g(\left|A \right|) x_1 + f(\left| A \right|) U^* x_1 \right\|^2
+\dfrac{1}{4}\left\| e^{i \theta} g(\left| D \right|) x_2+ 
f(\left| D \right|) V^* x_2 \right\|^2 \\
& \quad = \dfrac{1}{4} \left\|
\begin{bmatrix} 
e^{i \theta}g(\left| A \right|) \quad f(\left| A \right|) U^*
\end{bmatrix}
\begin{bmatrix}
x_1\\ x_1
\end{bmatrix}
\right\|^2 
+ \dfrac{1}{4} \left\|
\begin{bmatrix} 
e^{i \theta}g(\left| D \right|) \quad f(\left|D\right|) V^*
\end{bmatrix}
\begin{bmatrix}
x_2\\ x_2
\end{bmatrix}
\right\|^2 \\
& \quad\leq \dfrac{1}{2} 
\left\| \begin{bmatrix} 
e^{i \theta}g(\left| A \right|) \quad f(\left| A \right|) U^*
\end{bmatrix}\right\| ^2 \left\|x_1 \right\|^2 
+ \dfrac{1}{2} 
\left\| \begin{bmatrix} 
e^{i \theta}g(\left| D \right|)\quad f(\left| D \right|) V^* 
\end{bmatrix}\right\|^2\left\|x_2 \right\|^2.
\end{align*}
Let 
$
\alpha:= \left\| \begin{bmatrix} 
e^{i \theta}g(\left| A \right|) \quad f(\left| A \right|) U^*
\end{bmatrix}\right\|,$ and let $
\beta:= \left\| \begin{bmatrix} 
e^{i \theta}g(\left| D \right|)\quad f(\left| D \right|) V^* 
\end{bmatrix}\right\|.$ Clearly,
{\small 
\begin{equation*}
\max_{\left\| x_1\right\|^2 + \left\| x_2\right\|^2 = 1} \left( \alpha^2 \left\| x_1\right\|^2 + \beta^2 \left\| x_2\right\|^2 \right) 
= \max_{\theta \in [0, \frac{\pi}{2}]} \left(\alpha^2 \sin^2\theta + \beta^2 \cos^2\theta \right) 
= \max \left( \alpha^2, \beta^2 \right).
\end{equation*}
}
Hence,
{\footnotesize 
\begin{align*}
& Re\left\langle e^{i \theta} Tx, x \right\rangle \\
&\leq \frac{1}{2} \max 
\left( 
\left\| \begin{bmatrix} 
e^{i \theta}g(\left| A \right|) \quad f(\left| A \right|) U^*
\end{bmatrix}\right\|^2,
\left\| \begin{bmatrix} 
e^{i \theta}g(\left| D \right|)\quad f(\left| D \right|) V^* 
\end{bmatrix}\right\|^2\right) \\
& = \frac{1}{2} \max
\left(  \left\| \begin{bmatrix} 
e^{i \theta}g(\left| A \right|) \quad f(\left| A \right|) U^*
\end{bmatrix} 
\begin{bmatrix}
e^{-i \theta}g(\left| A \right|) \\ U f(\left| A \right|) \end{bmatrix} \right\|,
\left\| \begin{bmatrix} 
e^{i \theta}g(\left| D \right|)\quad f(\left| D \right|) V^* 
\end{bmatrix}
\begin{bmatrix}
e^{-i \theta}g(\left| D \right|) \\ V f(\left| D \right|) \end{bmatrix} \right\| \right) \\
&= \dfrac{1}{2} \max
\left(  \left\|g^2(\left| A \right|)+ f(\left| A \right|)U^* U f(\left| A \right|)\right\|, 
\left\|g^2(\left| D \right|) + f(\left| D \right|) V^*V f(\left| D \right|)\right\| \right) \\
& = \dfrac{1}{2}\max \left( \left\|f^2(\left| A \right|)+g^2( \left| A \right|)\right\|
, \left\|f^2(\left| D \right|)+ g^2(\left| D \right|) \right\|\right).
\end{align*}
}
Taking the supremum over all unit vectors $ x=(x_1, x_2)$ and using Lemma \ref{lemma 2.4} (a), yields that 
\begin{equation*}
w(T) \leq \dfrac{1}{2}\max \left( \left\|f^2(\left| A \right|)+g^2( \left| A \right|)\right\|
, \left\|f^2(\left| D \right|)+ g^2(\left| D \right|) \right\| \right).
\end{equation*} 
Therefore,
\begin{align*}
h(w(T)) & \leq \max \left( h\left(\left\| \dfrac{f^2(\left| A \right|)+g^2(\left| A \right|)}{2}\right\|\right) 
, h\left(\left\| \dfrac{ f^2(\left| D \right|)+ g^2(\left| D \right|) }{2} \right\| \right) \right) 
\tag{since \textit{h} is nondecreasing }\\
& \leq \dfrac{1}{2} \max \left( \left\| h\left( f^2(\left| A \right|)\right)+ h\left( g^2(\left| A \right|)\right) \right\| 
, \left\| h\left( f^2(\left| D \right|)\right)+ h\left( g^2(\left| D \right|)\right) \right\| \right)
\tag{by Lemma \ref{rooin1}.}
\end{align*}
\end{proof}
\begin{corollary}
Let $ A, D \in\mathbb{B}(\mathcal{H})$. Then 
\begin{equation}\label{inequality 3.2}
w^r\left(\begin{bmatrix}
A& 0 \\
0& D \end{bmatrix}\right) 
\leq \dfrac{1}{2} \max 
\left( \left\|\,\left| A\right|^{2r\alpha} + 
\left| A \right|^{2r(1-\alpha)} \right\|
,\left\|\,\left| D\right|^{2r\alpha} + 
\left| D \right|^{2r(1-\alpha)} \right\| \right),
\end{equation}
and in particular,
\begin{equation}\label{inequality 3.3}
w^r (A)\leq 
\dfrac{1}{2} \left\| \ \left|A\right|^{2r\alpha} + \left| A \right|^{2r(1-\alpha)} \ \right\|,
\end{equation}
for all $ r \geq 1$ and 
$ \alpha \in [0, 1]$. 
\end{corollary}
\begin{proof}
Take $ h(t)= t^r $, $f(t)= t^\alpha $, and 
$ g(t)= t^{1-\alpha}$ in inequality (\ref{inequality 3.1}). Inequality (\ref{inequality 3.3}) follows from (\ref{inequality 3.2})
by putting $ A=D $ and considering Lemma \ref{lemma 2.4} (b). 
\end{proof}
\begin{corollary} \label{corollary 3.3}
Let $ A, B, C $, and $ D \in\mathbb{B}(\mathcal{H})$. With the assumptions of Theorem \ref{theorem 3.1}, if 
$ Y=\begin{bmatrix}
A& B \\
C& D \end{bmatrix} $,
then
{\small 
\begin{align}\label{inequality 3.4}
\nonumber h\left( \dfrac{w(Y)}{2}\right) 
& \leq \dfrac{1}{4} \max \left( \left\| h\left( f^2(\left| A \right|)\right) + 
h\left(g^2(\left| A \right|)\right)\right\|
, \left\| h\left( f^2(\left| D \right| )\right) + 
h\left( g^2(\left| D \right|)\right)\right\| \right) \\
& \quad + \dfrac{1}{8} \left(\left\| h\left(f^2(\left| B \right|)\right)+ 
h\left( g^2(\left| B \right| )\right) \right\|
+ \left\| h\left( f^2(\left| C \right| ) \right) + 
h\left( g^2(\left| C \right| )\right) \right\| \right).
\end{align}
}
\end{corollary}
\begin{proof}
Use the triangular inequality together with the nondecreasingness and the convexity of $h$ and apply Theorems \ref{theorem 2.5} and \ref{theorem 3.1} to get the result.
\end{proof}
\begin{corollary}
Let $A, B \in\mathbb{B}(\mathcal{H})$. Then
\begin{align*}
& \max \left( w^r(A \pm B), w^r(A \pm iB)\right) \\
& \qquad\leq 2^{r-2} \left\|\, \left| A \right|^{2r\alpha} + 
\left| A \right|^{2r(1-\alpha)} \right\|
+ 2^{r-2}\left\|\, \left| B \right|^{2r\alpha} + 
\left| B\right|^{2r(1-\alpha)} \right\|,
\end{align*}
for all 
$ \alpha \in [0, 1]$ and $ r \geq 1 $.
\end{corollary}
\begin{proof}
Take in Corollary \ref{corollary 3.3}, $Y=\begin{bmatrix}
A& B \\
\pm B & A \\
\end{bmatrix}$, $ h(t) = t^r $, $f(t) = t^\alpha $, $ g(t)= t^{1- \alpha} \ (t \geq 0) $ to get
\begin{align*}
\left(\dfrac{ w \left( \begin{bmatrix}
A& B \\
\pm B & A \\
\end{bmatrix}\right)}{2} \right)^r 
& \leq \frac{1}{4} \left\| \left| A\right| ^{2 r \alpha} + \left| A\right| ^{2 r (1- \alpha)}\right\| \\
& + \frac{1}{8} \left( \left\| \left| B\right| ^{2 r \alpha} + \left| B \right| ^{2 r (1- \alpha)}\right\| + \left\| \left| \pm B\right| ^{2 r \alpha} + \left| \pm B \right| ^{2 r (1- \alpha)}\right\|\right) \\
& = \frac{1}{4} \left\| \left| A\right| ^{2 r \alpha} + \left| A\right| ^{2 r (1- \alpha)}\right\| + \frac{1}{4} \left\| \left| B \right| ^{2 r \alpha} + \left| B \right| ^{2 r (1- \alpha)}\right\|.
\end{align*}
Hence
\begin{equation*}
w^r \left( \begin{bmatrix}
A& B \\
\pm B & A \\
\end{bmatrix}  \right) 
\leq 2^{r-2} \left\| \left| A\right| ^{2 r \alpha} + \left| A\right| ^{2 r (1- \alpha)}\right\| + 2^{r-2} \left\| \left| B \right| ^{2 r \alpha} + \left| B \right| ^{2 r (1- \alpha)}\right\|.
\end{equation*}
Now, an application of (c), (d) of Lemma \ref{lemma 2.4} completes the proof.
\end{proof}
The next result reads as follows.
\begin{theorem}
Let 
$ T= \begin{bmatrix}
A& 0 \\
0& D \end{bmatrix} \in\mathbb{B}(\mathcal{H} \oplus \mathcal{H})$, and let $f$ and $g$ be non-negative continuous functions on $ [0, \infty )$ such that $ f(t)g(t)=t \ (t \geq 0) $.
Then for all non-negative nondecreasing convex functions $h$ on 
$ [0, \infty ) $,
{\footnotesize 
\begin{equation}\label{inequality 3.5}
h(w(T))
\leq \dfrac{1}{2} \max \left( 
\left\| h\left( f^2(\left| A \right| )\right) + 
h\left( g^2(\left| A^* \right| )\right) \right\|
,\left\| h\left( f^2(\left| D \right| )\right) + 
h\left( g^2(\left| D^* \right| )\right) \right\|
\right).
\end{equation}
}
\end{theorem}
\begin{proof}
Let $x$ be any unit vector in 
$\mathcal{H} \oplus \mathcal{H}$. We observe that 
\begin{align*}
h\left( \left|\right\langle Tx, x \left\rangle \right| \right) 
&\leq h \left( \left\langle f^2(\left| T \right|)x, x \right\rangle ^{\frac{1}{2}} \left\langle g^2(\left| T^* \right|)x, x \right\rangle^{\frac{1}{2}} \right)
\tag{by the mixed Cauchy--Schwarz inequality} \\
& \hspace{-0.5 cm}\leq h \left(\dfrac{\left\langle f^2(\left| T \right|)x, x \right\rangle + \left\langle g^2(\left| T^* \right|)x, x \right\rangle}{2} \right) \tag{by the Young inequality}\\
& \hspace{-0.5 cm} \leq \dfrac{1}{2} \left( h\left( \left\langle f^2(\left| T \right|)x, x \right\rangle\right)
+h\left( \left\langle g^2(\left| T^* \right|)x, x \right\rangle\right)\right) 
\tag{by the convexity of \textit{h}}\\
& \hspace{-0.5 cm} \leq \dfrac{1}{2} \left(\left\langle h \left( f^2(\left| T \right|)\right) x, x \right\rangle
+\left\langle h\left( g^2(\left| T^* \right|)\right) x, x \right\rangle \right) 
\tag{by Lemma \ref{lemma 2.2}} \\
& \hspace{-0.5 cm} = \dfrac{1}{2} \left\langle 
\begin{bmatrix}
h \left( f^2(\left| A \right| )\right)+ h \left( g^2(\left| A^* \right| ) \right) & 0 \\
0 & h \left( f^2(\left| D \right| ) \right)+ h \left( g^2(\left| D^* \right| )\right) 
\end{bmatrix} x,x \right\rangle.
\end{align*}
Taking the supremum over all unit vectors $x$, we reach the required result.
\end{proof}
\begin{corollary}
Let $ A, D \in\mathbb{B}(\mathcal{H})$. Then for all $ r \geq 1$ and 
$ \alpha \in [0, 1]$, 
\begin{equation}\label{inequality 3.6}
w^r \left( \begin{bmatrix}
A& 0 \\
0& D \end{bmatrix}\right) 
\leq \dfrac{1}{2} \max 
\left( \left\|\,\left| A\right|^{2r\alpha} + 
\left| A^* \right|^{2r(1-\alpha)} \right\|
,\left\|\,\left| D\right|^{2r\alpha} + 
\left| D^* \right|^{2r(1-\alpha)} \right\|
\right).
\end{equation}
In particular (see \cite[Theorem 1]{12}),
\begin{equation}\label{inequality 3.7}
w^r(A) \leq 
\dfrac{1}{2} \left\| \, \left|A\right|^{2r\alpha} + \left| A^* \right|^{2r(1-\alpha)} \ \right\|,
\end{equation}
and
\begin{equation} \label{inequality 888}
w^r(A) \leq 
\dfrac{1}{2} \left\| \, \left|A\right|^r + \left| A^* \right|^r \ \right\|.
\end{equation}
\end{corollary}
\begin{proof}
Inequality (\ref{inequality 3.6}) follows from inequality
(\ref{inequality 3.5}) by putting $ h(t)= t^r $, $f(t)= t^\alpha $ and $ g(t)= t^{1-\alpha}$. 
\end{proof}
\begin{remark}
Let $ A \in\mathbb{B}(\mathcal{H})$. For all $ r \geq 1$ and 
$ \alpha \in [0, 1]$, by using inequalities 
(\ref{inequality 3.3}) and (\ref{inequality 3.7}), we get
\begin{equation*}
w^r(A) \leq \dfrac{1}{2} \min \left( 
\left\| \, \left|A\right|^{2r\alpha} + \left| A^* \right|^{2r(1-\alpha)} \ \right\|,
\left\| \, \left|A\right|^{2r\alpha} + \left| A \right|^{2r(1-\alpha)} \ \right\|\right). 
\end{equation*} 
\end{remark}
The following theorem presents a generalization of inequality (\ref{inequality 1.5}) and the second inequality in 
(\ref{inequality 1.6}).
\begin{theorem}
Let 
$ T= \begin{bmatrix}
A& 0 \\
0& D \end{bmatrix} \in\mathbb{B}(\mathcal{H} \oplus \mathcal{H})$ and let $f$ and $g$ be non-negative continuous functions on $ [0, \infty) $ such that $ f(t)g(t)=t \ (t \geq 0) $.
Then for all non-negative nondecreasing convex functions $h$ on 
$ [0, \infty) $,
{\footnotesize 
\begin{align}\label{inequality 3.8}
\nonumber & h\left( w^r(T)\right) \\
& \qquad \leq \max \left( 
\left\| \dfrac{1}{p} h\left( f^{pr}(\left| A \right| )\right) + 
\dfrac{1}{q}h\left( g^{qr}(\left| A^* \right| )\right) \right\|
,\left\|\dfrac{1}{p} h\left( f^{pr}(\left| D \right| )\right) + 
\dfrac{1}{q}h\left( g^{qr}(\left| D^* \right| )\right) \right\|
\right),
\end{align}
}
where $p, q > 1 $ with 
$\dfrac{1}{p} + \dfrac{1}{q} =1$, and $ r \min(p, q) \geq 2$.
\end{theorem}
\begin{proof}
Without loss of generality, we can assume that $ p\geq q$.
Let $x$ be any unit vector in 
$\mathcal{H} \oplus \mathcal{H}$. Then
\begin{align*}
& h\left( \left|\right\langle Tx, x \left\rangle \right|^r \right)\\
&\leq h \left( \left\langle f^2(\left| T \right|)x, x \right\rangle ^{\frac{r}{2}} \left\langle g^2(\left| T^* \right|)x, x \right\rangle^{\frac{r}{2}} \right) 
\tag{by the mixed Cauchy--Schwarz inequality} \\
&\leq h \left( \dfrac{1}{p} \left\langle f^2(\left| T \right|)x, x \right\rangle ^{\frac{pr}{2}} + \dfrac{1}{q}\left\langle g^2(\left| T^* \right|)x, x \right\rangle^{\frac{qr}{2}} \right)
\tag{by the Young inequality}\\
& \leq h \left( \dfrac{1}{p} \left\langle f^{pr}(\left| T \right|)x, x \right\rangle + \dfrac{1}{q}\left\langle g^{qr}(\left| T^* \right|)x, x \right\rangle \right) 
\tag{by Lemma \ref{lemma 2.2} }\\
& \leq \dfrac{1}{p} h \left( \left\langle  f^{pr}(\left| T \right|) x, x \right\rangle\right)+ \dfrac{1}{q} h \left( \left\langle g^{qr}(\left| T^* \right|) x, x \right\rangle\right)
\tag{by the convexity of \textit{h}} \\
& \leq \dfrac{1}{p} \left\langle h \left( f^{pr}(\left| T \right|)\right) x, x \right\rangle + \dfrac{1}{q}\left\langle h \left(g^{qr}(\left| T^* \right|)\right) x, x \right\rangle
\tag{by Lemma \ref{lemma 2.2}} \\
& = \left\langle 
\begin{bmatrix}
\dfrac{1}{p}h\left(f^{pr}\left(\left|A\right|\right)\right) +\dfrac{1}{q}h\left( g^{qr} \left(\left|A^*\right|\right)\right)& 0\\
0 & \dfrac{1}{p}h \left( f^{pr}(\left| D \right|)\right) +\dfrac{1}{q} h\left( g^{qr}(\left| D^* \right| )\right) 
\end{bmatrix} x,x \right\rangle. 
\end{align*}
Take the supremum over all unit vectors $x$ to reach the required result.
\end{proof}
In the next corollary, inequality (\ref{inequality 3.9}) is a generalization of the second inequality in (\ref{inequality 1.6}).
\begin{corollary}\label{corollayr 3.9}
Let $ A, D \in\mathbb{B}(\mathcal{H})$, $p, q > 1$ with 
$\dfrac{1}{p} + \dfrac{1}{q} =1$, and $ r \min(p, q) \geq 2 $. Then
{\small 
\begin{equation*}
w^{2r}\left( \begin{bmatrix}
A& 0 \\
0& D \end{bmatrix}\right)
\leq \max \left( 
\left\| \dfrac{1}{p}\left| A \right|^{2pr \alpha} + 
\dfrac{1}{q} \left| A^* \right|^{2qr(1- \alpha)} \right\|
,\left\|\dfrac{1}{p}\left| D \right|^{2pr \alpha} + 
\dfrac{1}{q} \left| D^* \right|^{2qr(1- \alpha)} \right\|
\right),
\end{equation*}
}
and in particular,
\begin{equation}\label{inequality 3.9}
w^{2r}(A) \leq \left\| \dfrac{1}{p}\left| A \right|^{2pr \alpha} + 
\dfrac{1}{q} \left| A^* \right|^{2qr(1- \alpha)} \right\|,
\end{equation}
for all 
$ \alpha \in [0, 1]$.
\end{corollary}

Note that, if, in Corollary \ref{corollayr 3.9}, we take 
$ \alpha = \dfrac{1}{p}$, then
 \begin{equation*}
w^{2r}\left( \begin{bmatrix}
A& 0 \\
0& D \end{bmatrix}\right) 
\leq \max \left( 
\left\| \dfrac{1}{p}\left| A \right|^{2r} + 
\dfrac{1}{q} \left| A^* \right|^{2r} \right\|
,\left\|\dfrac{1}{p}\left| D \right|^{2r}+ 
\dfrac{1}{q} \left| D^* \right|^{2r} \right\|
\right).
\end{equation*}
In particular (see \cite[Theorem 2]{12}),
\begin{equation*}
w^{2r}(A) \leq \left\| \dfrac{1}{p}\left| A \right|^{2r} + 
\dfrac{1}{q} \left| A^* \right|^{2r} \right\|. 
\end{equation*}
In the next corollary, inequality (\ref{inequality 3.10}) is a generalization of inequality (\ref{inequality 1.5}).
\begin{corollary}\label{corollary 3.10}
Let $ A, D \in\mathbb{B}(\mathcal{H})$, $p, q > 1$ with 
$\dfrac{1}{p} + \dfrac{1}{q} =1$, and $ r \min(p, q) \geq 2 $. Then 
\begin{equation*}
w^{r}\left( \begin{bmatrix}
A& 0 \\
0& D \end{bmatrix}\right)
\leq \max \left( \left\| \dfrac{1}{p}\left| A \right|^{pr \alpha} + 
\dfrac{1}{q} \left| A^* \right|^{qr(1-\alpha)} \right\|
, \left\|\dfrac{1}{p}\left| D \right|^{p r \alpha} + 
\dfrac{1}{q} \left| D^* \right|^{qr(1-\alpha)} \right\|\right),
\end{equation*}
and in particular (see \cite[Corollary 3]{15}),
\begin{equation}\label{inequality 3.10}
w^r(A) \leq \left\| \dfrac{1}{p}\left| A \right|^{p r \alpha} + 
\dfrac{1}{q} \left| A^* \right|^{qr(1-\alpha)} \right\|,
\end{equation}
for all $ \alpha \in [0, 1]$.
\end{corollary}

Note that, if in Corollary \ref{corollary 3.10} we take 
$ \alpha= \dfrac{1}{p}$, then 
\begin{equation*}
w^{r}\left( \begin{bmatrix}
A& 0 \\
0& D \end{bmatrix}\right)
\leq \max \left( \left\| \dfrac{1}{p}\left| A \right|^r + 
\dfrac{1}{q} \left| A^* \right|^r \right\|
,\left\|\dfrac{1}{p}\left| D \right|^r +
\dfrac{1}{q} \left| D^* \right|^r \right\| \right).
\end{equation*}
In particular,
\begin{equation*}
w^r(A) \leq \left\| \dfrac{1}{p}\left| A \right|^r + 
\dfrac{1}{q} \left| A^* \right|^r \right\|,
\end{equation*}
which is a generalization of inequality (\ref{inequality 888}).\\
In the next theorem, we give a generalization of the second inequality in (\ref{inequality 1.7}).
\begin{theorem} \label{theorem 3.13}
Let 
$ T= \begin{bmatrix}
A& 0 \\
0& D \end{bmatrix} \in\mathbb{B}(\mathcal{H} \oplus \mathcal{H})$, $ r \geq 2 $, and $p, q > 1$ with $\dfrac{1}{p} + \dfrac{1}{q}=1$. If 
$f_1$, $g_1$, $f_2$, and $g_2$ are non-negative continuous functions on $[0, \infty)$ such that $f_1(t)g_1(t)=f_2(t)g_2(t)=t \ (t \geq 0)$, then 
\begin{equation}\label{inequality 3.11}
w^r(T) \leq \dfrac{1}{2}
{\max}^{\frac{1}{p}} \left(\alpha, \beta \right) 
{\max}^{\frac{1}{q}} \left(\gamma, \delta \right),
\end{equation}
and
\begin{equation}\label{inequality 3.12}
w^r(T) \leq \dfrac{1}{2}
{\max}^{\frac{1}{p}}\left(\alpha^{'}, \beta^{'}\right) 
{\max}^{\frac{1}{q}}\left(\gamma^{'},\delta^{'}\right), 
\end{equation}
where\\
{\scriptsize
$ \alpha= \left\| f_1^{rp}( \left| {\rm Re} A + {\rm Im} A \right| ) + f_2^{rp} ( \left| {\rm Re} A - {\rm Im} A\right| )\right\|, \qquad
\beta= \left\| f_1^{rp} ( \left| {\rm Re} D + {\rm Im} D \right| ) + f_2^{rp} ( \left| {\rm Re} D - {\rm Im} D \right| )\right\| $,\\
$ \gamma= \left\| g_1^{rq} ( \left| {\rm Re} A + {\rm Im} A\right| ) + g_2^{rq} (\left| {\rm Re} A - {\rm Im} A \right| )\right\|,\qquad
\delta= \left\| g_1^{rq} (\left| {\rm Re} D + {\rm Im} D\right| ) + g_2^{rq} ( \left| {\rm Re} D - {\rm Im} D\right| )\right\|,$
}\\
and \\
{\scriptsize
$ \alpha^{'}= \left\| f_1^{rp} (\left| {\rm Re} A + {\rm Im} A\right| )+ g_2^{rp} ( \left| {\rm Re} A - {\rm Im} A\right| )\right\|, \qquad
\beta^{'}= \left\| f_1^{rp} ( \left| {\rm Re} D + {\rm Im} D \right| ) + g_2^{rp} ( \left| {\rm Re} D - {\rm Im} D\right| ) \right\| $,\\
$ \gamma^{'}= \left\| g_1^{rq} (\left| {\rm Re} A + {\rm Im} A\right| )+ f_2^{rq} ( \left| {\rm Re} A - {\rm Im} A\right| )\right\|, \qquad
\delta^{'}= \left\| g_1^{rq} ( \left| {\rm Re} D + {\rm Im} D \right| ) + 
f_2^{rq} ( \left| {\rm Re} D - {\rm Im} D\right| ) \right\| $.
}
\end{theorem}
\begin{proof}
Assume that $ T = T_1 + iT_2 $ is the Cartesian decomposition of $T$ and that $x$ is any unit vector in 
$\mathcal{H} \oplus \mathcal{H}$. Then
\begin{align*}
\left| \left\langle Tx, x\right\rangle\right|^r 
& = \left|\left\langle \left( T_1 + iT_2\right)x, x \right\rangle \right|^r \\ 
&=\left(\left\langle T_1 x, x \right\rangle^2 + \left\langle T_2 x,x \right\rangle^2 \right)^{\frac{r}{2}} \\
&= 2^{-\frac{r}{2}}\left(\left\langle(T_1 + T_2)x, x\right\rangle^2 + \left\langle (T_1 - T_2)x,x\right\rangle^2 \right)^{\frac{r}{2}} \\
&\leq 2^{-\frac{r}{2}} 2^{\frac{r}{2}-1}\left(\left| \left\langle(T_1 + T_2)x, x \right\rangle \right| ^r +\left|\left\langle (T_1 - T_2)x,x \right\rangle \right| ^r \right)
\tag{by the convexity of $ t^{\frac{r}{2}}$ for $r \geq 2$} \\
& \leq\frac{1}{2}\left( \left\langle \left| T_1 + T_2 \right| x, x \right\rangle ^r +  \left\langle \left| T_1 - T_2 \right| x, x \right\rangle ^r \right) \tag{by the convexity of $ \left|t \right|$} \\
&\leq\frac{1}{2} \left( \left\langle f_1^2(\left| T_1 + T_2\right|) x, x\right\rangle^{\frac{r}{2}} 
\left\langle g_1^2(\left| T_1 + T_2\right|) x, x\right\rangle^{\frac{r}{2}}\right)\\
& \quad +\frac{1}{2} \left( \left\langle f_2^2(\left| T_1 - T_2\right| )x, x\right\rangle^{\frac{r}{2}} 
\left\langle g_2^2(\left| T_1 - T_2\right|) x, x\right\rangle^{\frac{r}{2}}
\right) \tag{by the mixed Cauchy--Schwarz inequality}\\
&\leq\frac{1}{2} \left( \left\langle f_1^r(\left|T_1 + T_2\right|) x, x\right\rangle
\left\langle g_1^r(\left| T_1 + T_2\right|) x, x\right\rangle\right)\\
& \quad +\frac{1}{2} \left( \left\langle f_2^r(\left| T_1 - T_2\right| )x, x\right\rangle
\left\langle g_2^r(\left| T_1 - T_2\right|) x, x\right\rangle
\right) 
\tag{by Lemma \ref{lemma 2.2}}\\
&\leq \frac{1}{2} \left(\left\langle f_1^{rp}\left(\left|T_1 + T_2\right| \right)x,x \right\rangle 
+\left\langle f_2^{rp}\left(\left|T_1 - T_2\right|\right)x,x \right\rangle 
\right)^{\frac{1}{p}}\\
& \quad \times \left(\left\langle g_1^{rq}\left( \left|T_1 + T_2\right|\right)x,x \right\rangle 
+\left\langle g_2^{rq}\left( \left|T_1 - T_2\right|\right)x,x \right\rangle 
\right)^{\frac{1}{q}}.\tag{by the H\"{o}lder inequality and Lemma \ref{lemma 2.2}} 
\end{align*}
Therefore,
{\scriptsize 
\begin{align*}
& \left| \left\langle Tx, x\right\rangle\right|^r \\
& \leq \dfrac{1}{2}\left\langle \begin{bmatrix}
f_1^{rp}( \left| {\rm Re} A + {\rm Im} A\right| ) + f_2^{rp}(\left| {\rm Re} A - {\rm Im} A\right| )& 0 \\
0& f_1^{rp}( \left| {\rm Re} D + {\rm Im} D\right| )+ f_2^{rp}( \left| {\rm Re} D - {\rm Im} D\right| )
\end{bmatrix} x, x \right\rangle^{\frac{1}{p}}\\
& \quad \times \left\langle \begin{bmatrix}
g_1^{rq}( \left| {\rm Re} A + {\rm Im} A\right| ) + g_2^{rq}( \left| {\rm Re} A - {\rm Im} A\right| )& 0 \\
0& g_1^{rq}(\left| {\rm Re} D + {\rm Im} D\right| )+ g_2^{rq}(\left| {\rm Re} D - {\rm Im} D\right| )
\end{bmatrix} x, x \right\rangle^{\frac{1}{q}}.
\end{align*}
}
Take the supremum over all unit vectors $x$ to get inequality (\ref{inequality 3.11}). Inequality
(\ref{inequality 3.12}) is obtained by a similar reasoning.
\end{proof}
\begin{corollary}
Let $ A \in\mathbb{B}(\mathcal{H})$ with the Cartesian decomposition $ A= B+iC $, and let $ r \geq 2 $. With the assumptions of Theorem \ref{theorem 3.13}, 
{\small 
\begin{equation*}
w^r(A) \leq \dfrac{1}{2} \left\|f_1^{rp}\left(\left|B+C\right| \right)+ f_2^{rp}\left(\left|B-C\right|\right)\right\|^{\frac{1}{p}}
\left\|g_1^{rq}\left(\left|B+C\right|\right)+ g_2^{rq}\left(\left|B-C\right|\right)\right\|^{\frac{1}{q}}
\end{equation*}
}
and
{\small 
\begin{equation}\label{inequality 3.13}
w^r(A) \leq \dfrac{1}{2} \left\|f_1^{rp}\left(\left|B+C\right| \right)+ g_2^{rp}\left(\left|B-C\right|\right)\right\|^{\frac{1}{p}}
\left\|g_1^{rq}\left(\left|B+C\right|\right)+ f_2^{rq}\left(\left|B-C\right|\right)\right\|^{\frac{1}{q}}, 
\end{equation}
}
where $p, q > 1$ with $\dfrac{1}{p} + \dfrac{1}{q}=1$.
\end{corollary}
\begin{corollary}
Let $ A \in\mathbb{B}(\mathcal{H})$ with the Cartesian decomposition $ A= B+iC $. Then for all $ \alpha \in [0, 1]$, $ r \geq 2 $ and $p, q > 1$ with $\dfrac{1}{p} + \dfrac{1}{q}=1$,
{\small 
\begin{equation}\label{inequality 3.15}
w^r(A) \leq \dfrac{1}{2} \left\| \left| B+C\right|^{r p \alpha} +
\left| B-C\right|^{r p (1- \alpha)} \right\|^{\frac{1}{p}} 
\left\| \left| B+C\right|^{r q \alpha} 
+ \left|B-C\right|^{r q (1- \alpha)} \right\|^{\frac{1}{q}}, 
\end{equation}
}
which is a generalization of the second inequality in 
(\ref{inequality 1.7}).
\end{corollary}
\begin{proof}
Take $ f_1(t)=f_2(t)= t^{\alpha} $, and $ g_1(t)=g_2(t)=t^{1- \alpha} $ in inequality (\ref{inequality 3.13}).
The second inequality in (\ref{inequality 1.7}) 
follows from inequality (\ref{inequality 3.15}) by putting $p=q =2$ and $ \alpha = \dfrac{1}{2}$.
\end{proof}
\begin{remark}
We end our work by mentioning that all inequalities in this paper are sharp. This fact comes from the sharpness of the second inequality of (\ref{inequality 1.2}). For example, if in Theorem \ref{theorem 2.5}, we take $h(t)= t$, $f(t)=g(t)= \sqrt{t} \ (t\geq 0)$ and $ B=C$, then we get
\begin{equation*}
\left( w(B)= \right)  w\left(\begin{bmatrix}
0& B \\
B& 0 \end{bmatrix} \right) \leq  \left\| B \right\|.
\end{equation*}
Also, if in Theorem \ref{theorem 2.9} we choose $ f_1(t) = f_2(t) = g_1(t) = g_2(t) = \sqrt{t} \ (t\geq0)$ , $r=2$ , $ p= q = 2$ and $ C=B=B^* $, we obtain 
\begin{equation*}
\left( w^2(B)= \right)  w^2\left(\begin{bmatrix}
0& B \\
B& 0 \end{bmatrix} \right) \leq  \left\| B \right\|^2.
\end{equation*}
The sharpness of the other inequalities is handled in the same manner.
\end{remark}

\end{document}